\documentclass[12pt]{amsart}
\usepackage{amssymb, amsfonts, amsthm}
\usepackage{graphicx}

\numberwithin{equation}{section}

\newtheorem{theorem}{Theorem}[section]
\newtheorem{lemma}[theorem]{Lemma}
\newtheorem{corollary}[theorem]{Corollary}

\newtheorem{definition}[theorem]{Definition}

\theoremstyle{definition}

\newcommand{\A}{{\mathcal A}}

\newcommand{\es}{{\mathcal S}}

\newcommand{\D}{{\mathbb D}}

\newcommand{\CC}{{\mathcal C}}

\newcounter{minutes}
\divide\time by 60
\newcounter{hours}
\multiply\time by 60 \addtocounter{minutes}{-\time}
\begin{document}

\title[Maximal area integral problem]
{On maximal area integral problem for analytic functions in the 
starlike family}

\author{S. K. Sahoo${}^*$}
\address{S. K. Sahoo, Discipline of Mathematics,
Indian Institute of Technology Indore,
Indore 452 017, India}
\email{swadesh@iiti.ac.in}

\author{N. L. Sharma}
\address{N. L. Sharma, Discipline of Mathematics,
Indian Institute of Technology Indore,
Indore 452 017, India}
\email{sharma.navneet23@gmail.com}

\thanks{${}^*$ The corresponding author}

\begin{center}
\texttt{File:~\jobname .tex, printed: \number\day-\number\month-\number\year,
\thehours.\ifnum\theminutes<10{0}\fi\theminutes}
\end{center}

\thispagestyle{empty}

%%%%%%%%%%%%%%%%%%%%%%%%%%%%%%%%%
%%%%%%%%%%% SECTION 1 %%%%%%%%%%%
%%%%%%%%%%%%%%%%%%%%%%%%%%%%%%%%%

\begin{abstract}
For an analytic function $f$ defined on the unit disk $|z|<1$, let
$\Delta(r,f)$ denote the area of the image of the subdisk $|z|<r$ under $f$, where
$0<r\le 1$. In 1990, Yamashita conjectured that $\Delta(r,z/f)\le \pi r^2$ for convex functions $f$
and it was finally settled in 2013 by Obradovi\'{c} and et. al.. 
In this paper, we consider a class of analytic functions in the unit disk satisfying the
subordination relation $zf'(z)/f(z)\prec (1+(1-2\beta)\alpha z)/(1-\alpha z)$ for
$0\le \beta<1$ and $0<\alpha\le 1$. We prove Yamashita's conjecture problem for functions 
in this class, which solves a partial solution to an open problem posed by Ponnusamy
and Wirths.\\


\end{abstract}

\maketitle

\section{\bf Introduction, Preliminaries, and Main result}

The univalent function has been the central object in the study of 
geometric function theory. Some of its natural
geometric families act a prominent role in the theory of univalent 
functions \cite{Dur83,Good83,POM75} and their geometric
properties. For instance, the classes of starlike, convex and 
close-to-convex, to name just a few. These classes have been 
familiarized and studied by many authors. It is interesting 
to observe that we can obtain many of their analytic properties
by an integrated method. Study of various subclasses of the class 
of starlike functions have been appreciated by several authors. 
The class of starlike
functions of order $\beta\,(0\leq \beta<1)$ was generated by 
Robertson~\cite{Rob36} and has been then studied by
Schild~\cite{Sch58} and Merkes~\cite{Mer62}. Marx~\cite{Mar33} and 
Strohacker~\cite{Str33} proved that if $f(z)$ maps
the unit disk onto a convex domain, then $f(z)$ is starlike of order 
$1/2.$ Gabriel~\cite{Gab55} showed that
the class of starlike functions of order $1/2$ played an important 
role in the solution of differential equations. 
In 1968, Padmanabhan~\cite{Pad68} discussed a different subfamily for the order of starlikeness.
In this paper, we introduce a more general family than the family studied by Padmanabhan.

Define by $\mathbb{D}_r:=\{z\in\mathbb{C}:\,|z|<r\}$, the disk
of radius $r$ centred at the origin. The unit disk is then
defined by $\D:=\D_1$.
Let $\A$ denote the family of all functions $f(z)$ analytic in $\mathbb{D}$ 
and normalized so that $f(0) = 0 = f'(0)-1$, i.e. $f\in\A$ has the power series
representation $f(z)=z+\sum_{n=2}^\infty a_n z^n$. Denote by $\es$, the class of univalent functions $f\in \A$.
The Gaussian hypergeometric function ${}_2F_1(a,b;c;z)$ is defined by the series
$$
1+\sum_{n=1}^{\infty}\frac{(a)_n(b)_n}{(c)_n(1)_n}z^n, \quad |z|<1,
$$
where $a,b$ and $c$ are complex numbers with $c$ is neither zero nor a negative integer. Clearly, the shifted function
$z{}_2F_1(a,b;c;z)$ belongs to $\A$.
The notation $(a)_n$ denotes the shifted factorial and it is defined by
$$
(a)_0 = 1, (a)_n = a(a + 1) \cdots (a + n-1) =\frac{\Gamma (a+n)}{\Gamma (a)},\quad n\geq 1.
$$
Here, $\Gamma$ stands for the usual gamma function. If either (or both) of 
$a$ and $b$ is (are) zero or a 
negative integer(s), then the series terminates.

For two analytic functions $f$ and $g$ in $\D$, we say that $f$ is {\em subordinate} to $g$
if 
$$f(z)=g(w(z)),\quad |z|<1,
$$
for some analytic function $w$ in $\D$ with $w(0)=0$ and $|w(z)|<1.$ 
We express this symbolically by $f \prec g$. In particular, if $g$ is univalent in $\D$,  $f(0)=g(0)$
and $f(\D)\subset g(\D)$ then $f \prec g$. For instance, one can easily see that 
$1/(1+z)\prec (1+z)/(1-z)$, $z\in \D$.

We denote by $\es t(\beta)$, the well-known class of starlike functions of order $\beta$.
Analytically, for $f\in \es, $ the starlike functions are characterized by the condition 
$ {\rm Re}\,\left(zf'(z)/f(z)\right)>\beta,$ where $ 0\leq \beta <1,$ i.e.
$f$ has the subordination property,
$$
\frac{zf'(z)}{f(z)} \prec \frac{1+(1-2\beta) z}{1- z},\quad z\in \D,\, 0\leq \beta<1.
$$
The class $\es t:=\es t(0)$ is the class of starlike functions, i.e.
$f\in \es$ is starlike with respect to the origin, i.e. $tw\in f(\D)$ 
whenever $t\in [0, 1]$ and $w\in f(\D)$.
 
Suppose that $f(z)$ is a function analytic in the unit disk $\D$. For $0<r\le 1$, 
we denote by $\Delta(r,f)$, the area of the image of the disk $\mathbb{D}_r$ under $f(z)$. Thus, 
$$
\Delta(r,f)=\iint_{\mathbb{D}_r}|f'(z)|^2 \, dxdy\quad (z=x+iy).
$$
Computing this area is known as the {\em area problem for the function of type $f$.}
The classical Parseval-Gutzmer formula for a function $f(z)=\sum_{n=0}^\infty a_nz^n$
analytic in $\overline{\mathbb{D}}_r$ states that
$$\frac{1}{2\pi} \int_0^{2\pi}|f(re^{i\theta})|^2\,d\theta = \sum_{n=0}^\infty |a_n|^2r^{2n}.
$$
By means of this formula, 
since
$f'(z)=\sum_{n=1}^\infty n a_nz^{n-1}$, we find
$$
\Delta(r,f)=\pi \sum_{n=1}^\infty n |a_n|^2r^{2n}.
$$
We call $f$ a Dirichlet-finite function if $\Delta(1,f)$, the area covered by the mapping
$z \to f(z)$ for $|z| < 1$, is finite.
Our interest in this paper was originated by the work of Yamashita~\cite{Yam90} and 
Ponnusamy et. al.~\cite{OPW13,OPW14,PW13}.
Yamashita \cite{Yam90} initially conjectured that
$$\displaystyle \max_{f\in \CC}\Delta \left(r,\frac{z}{f}\right)=\pi r^2
$$
for each $r,\, 0<r\leq 1,$ and the maximum is attained only by the rotations of the function
$l(z)=z/(1-z).$
Here $\CC$ denotes the class of convex functions i.e $f\in\es$ such that $f(\D)$ is convex.
This conjecture was recently settled in \cite{OPW13}. In fact the conjecture has been solved
for a wider class of functions (the class of starlike functions of order $\beta$, $0\le \beta<1$), 
which includes the class $\mathcal{C}$; see also Corollary~\ref{cor3}.
In \cite{PW13}, the Yamashita conjecture problem for the class of $\phi$-spirallike functions of order 
$\beta ~(0\leq \beta <1)$ and $\phi \in (-\pi /2,\pi /2)$ have also been settled 
(see \cite{Lib67} for the definition of $\phi$-spirallike function).
Recent work in this direction can also be found in \cite{OPW14}. There are several other classes
of analytic univalent functions having interesting geometric properties for which solution of the 
Yamashita conjecture problem would be of interesting to readers working in this field.

Our objective in this paper is to give a partial solution to a problem posed in \cite{PW13} 
by considering the following subfamily of the family of starlike functions introduced by 
Padmanabhan \cite{Pad68}.
\begin{definition}\label{def1}
A function $f\in\A$  is said to be in  $\es (\alpha)$, $0<\alpha \le 1$, if
$$\left|\left(\frac{zf'(z)}{f(z)}-1\right){\Big/} \left(\frac{zf'(z)}{f(z)}+1\right) \right|<\alpha,
~~\mbox{ equivalently, }~~
\frac{zf'(z)}{f(z)} \prec \frac{1+\alpha z}{1-\alpha z},
$$
for all $z\in \D$.
\end{definition}
It is evident to see that $\es (\alpha)\subset \es t(\beta).$  
One can also verify that $k_{\beta}(z):=z/(1-z)^{2(1-\beta)}$ belongs to $\es t(\beta),$
whereas, the function  $k_{\alpha}(z):=z/(1-\alpha z)^{2}\in \es (\alpha).$
In this paper, we prove the Yamashita conjecture for functions in a more general family than $\es(\alpha)$.
In particular, the conjecture will also follow for functions in $\es(\alpha)$. The generalization is
now defined below.
 
\begin{definition}\label{def2}
A function $f\in\A$ is said to belong to the class $\es (\alpha, \beta)$, $0<\alpha\leq 1$, $0\leq \beta<1$, if
$$\left|\left(\frac{zf'(z)}{f(z)}-1\right){\Big/} \left(\frac{zf'(z)}{f(z)}+1-2\beta\right) \right|
<\alpha,
$$
i.e.,
\begin{equation}\label{eq1}
\frac{zf'(z)}{f(z)} \prec \frac{1+(1-2\beta)\alpha z}{1-\alpha z},
\end{equation}
for all $z\in \D$.
\end{definition}
A general form of this definition is earlier introduced by Aouf (see \cite[Definition~2]{Aou88}).
Definition of $\es (\alpha, \beta)$ says that the domain values of $\displaystyle {zf'(z)}/{f(z)}$
lie in the disk of radius $ \displaystyle {2\alpha(1-\beta)}/({1-\alpha^2})$
centred at $({1+\alpha^2(1-2\beta)})/({1-\alpha^2})$.
We see that if $\beta =0,$ then Definition \ref{def2} turns into Definition \ref{def1}.
 The function 
\begin{equation}\label{eq2}
 k_{\alpha,\beta}(z):=z(1-\alpha z)^{-2(1-\beta)},\quad  z\in \D,\,  0<\alpha\leq 1, \, 0\leq \beta<1,
\end{equation}
belongs to the family $\es (\alpha, \beta)$ and in this context it plays the role of extremal function 
for $\es (\alpha, \beta)$. Also, one notes that
$$\es (\alpha, \beta)\subset \es (\alpha)
\subset \es t(\beta) \subset \es t;\, 
\es (1, \beta)=\es t(\beta),\, \es (\alpha, 0)=\es (\alpha) ~\mbox{and}~ \es (1)=\es t.
$$
Consequently,
\begin{equation}\label{eq3}
 k_{1,\beta}(z)=k_{\beta}(z),\, k_{\alpha,0}(z)=k_{\alpha}(z),\, k_{1,0}(z)=k_{1}(z)=k(z).
\end{equation}
The images of several sub-disks $\D_r$, $0<r\le 1$, under the functions $k_{\alpha,\beta}(z)$ for different values 
of $\alpha$ and $\beta$ have been described in Figure~1.

\begin{figure}
\begin{minipage}[b]{0.5\textwidth}
\includegraphics[width=6.7cm]{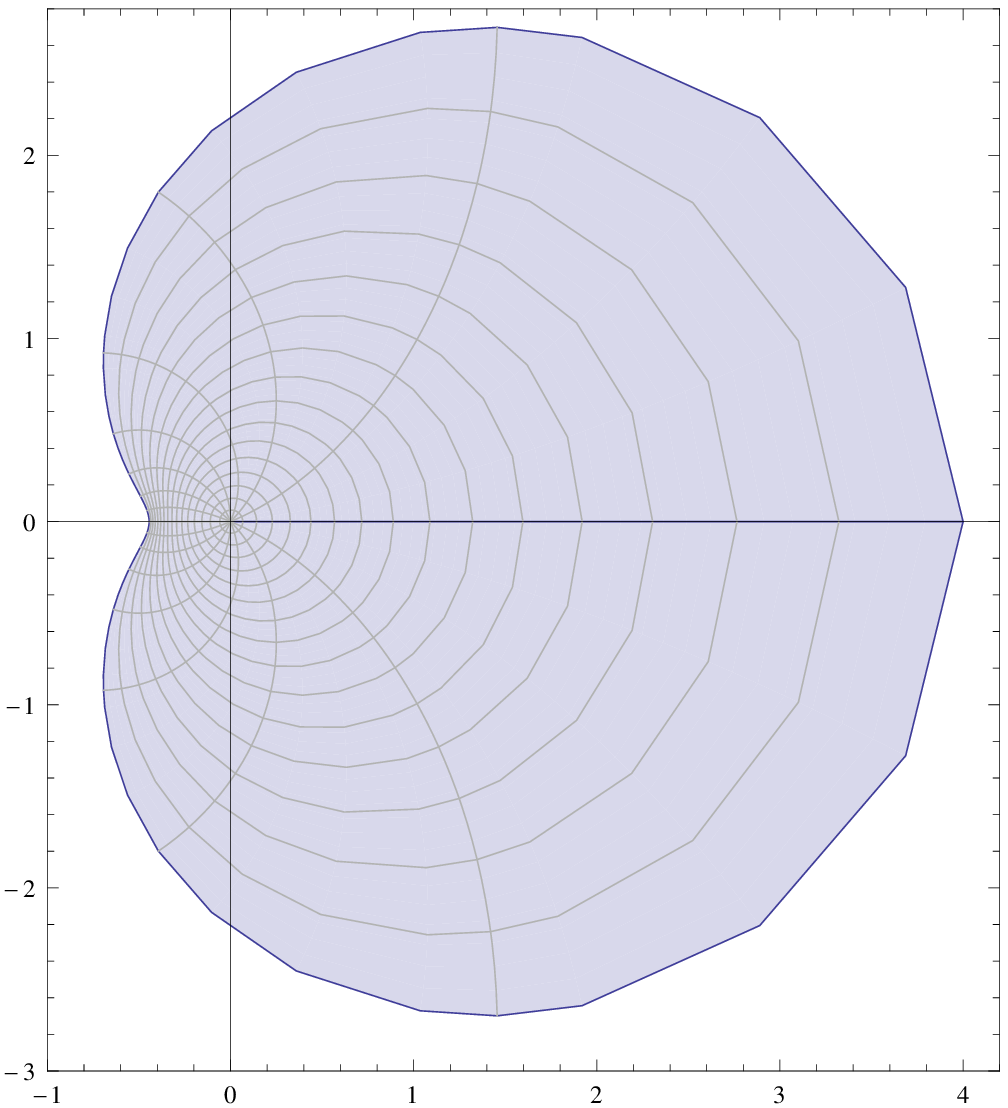}
\hspace*{1cm}The image domain $k_{1/2,0}(\D)$
\end{minipage}
\begin{minipage}[b]{0.45\textwidth}
\includegraphics[width=6.7cm]{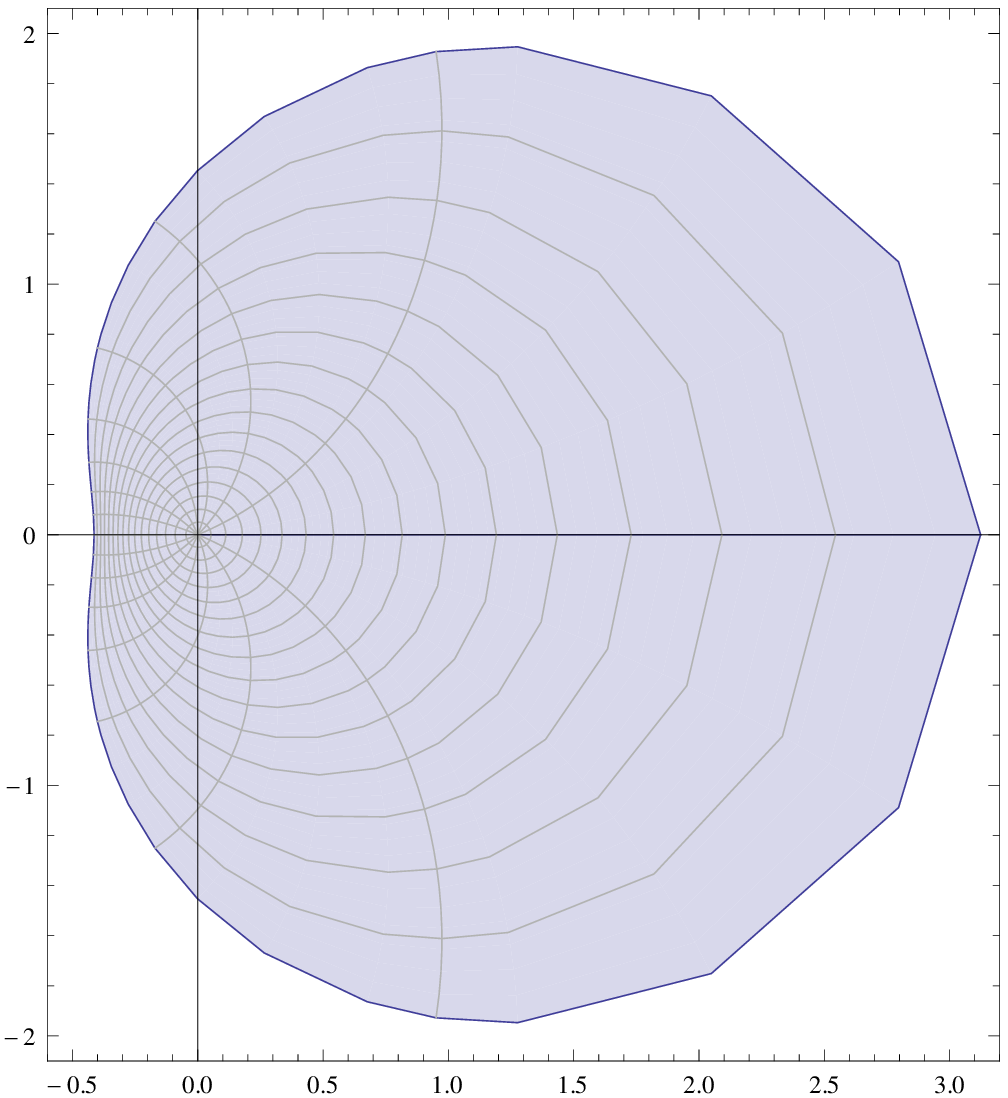}
\hspace*{.3cm} The image domain $k_{4/5,1/3}(\D_{0.8})$
\end{minipage}
\caption{Images of the disks $\D$ and $\D_{0.8}$ under $k_{1/2,0}$ and $k_{4/5,1/3}$.}
\end{figure}

In this article, our main aim is to examine the maximum area problem for the functions of type $(z/f)$ 
so-called the Yamashita conjecture problem, when $f\in \es (\alpha,\beta)$. By looking into the behavior of 
$z/k_{\alpha,\beta}(z)$ (see Section~3), we expect the following theorem whose proof is given in Section~3.

\begin{theorem}\label{thm1}{\bf (Main Theorem)}
Let $0<\alpha\leq 1$ and $0\leq \beta <1$. If $f(z)\in \es (\alpha,\beta)$
and $z/f(z)$ is a non-vanishing analytic function in $\D$,
then we have the maximal area
\begin{align*}
 \displaystyle \max_{f\in \es (\alpha,\beta)} \Delta \left(\rho,\frac{z}{f}\right)
 & = 4\pi{\alpha}^2(1-\beta)^2{\rho}^2 {}_2F_1(2\beta-1,2\beta-1;2;{\alpha}^2{\rho}^2),\quad |z|<\rho\\
 & =:A_{\alpha,\beta}(\rho)
 \end{align*} 
  for all $\rho, \, 0<\rho\leq 1$. The maximum is attained only by the rotations of the function 
$ k_{\alpha,\beta}(z)$ defined by $(\ref{eq2}).$
\end{theorem}
 This generalizes the main results which are discussed in \cite{OPW13} and \cite{Yam90}.
 
In Section~2, we prepare some basic results and use them to prove our main theorem in Section~3.

\section{\bf Preparatory Results}  

If $f\in \es$ then $z/f(z)$ is non-vanishing in $\D$. Hence, it can be described as Taylor's series of the form
\begin{equation}\label{eq-1-sec.2}
\frac{z}{f(z)}=1+\sum_{n=1}^{\infty}b_nz^n, \, z\in \D.
\end{equation}

We first derive a coefficient estimate in series form for a function $f$ of the form (\ref{eq-1-sec.2})
when $f\in \es (\alpha,\beta).$
\begin{lemma}\label{lem1}
Let $0<\alpha\leq 1$, $0\leq \beta <1$, and $f(z)\in \es (\alpha,\beta)$. If
$g(z)$ is a non-vanishing analytic function in $\D$ of the form $(\ref{eq-1-sec.2})$, then it
necessarily satisfies the coefficient inequality
$$\sum_{k=1}^{\infty} \left(k^2-(k-2(1-\beta))^2 \alpha^2 \right)|b_k|^2
\leq 4(1-\beta)^2 \alpha^2.
$$
\end{lemma}
\begin{proof}
Let $g(z)={z}/{f(z)}$ be of the form (\ref{eq-1-sec.2}). Note that the logarithmic derivative gives
$$\frac{zg'(z)}{g(z)}=1-\frac{zf'(z)}{f(z)}.
$$ 
Since $f(z)\in \es (\alpha,\beta)$, the subordination relation (\ref{eq1}) says that 
there exists an analytic function $ w:\D \rightarrow \overline{\D}$  with $w(0)=1$ such that
$$\frac{zf'(z)}{f(z)}=\frac{1+\alpha(1-2\beta)zw(z)}{1-\alpha zw(z)}, \, z\in \D , 
$$
and hence
$$\frac{g'(z)}{g(z)}=\frac{-2\alpha(1-\beta)w(z)}{1-\alpha zw(z)}.
$$
Writing this in series form, we get
$$\sum_{k=1}^{\infty}kb_kz^{k-1}=\alpha\left(-2(1-\beta)+\sum_{k=1}^{\infty}(k-2(1-\beta))b_kz^k\right)w(z).
$$
After a minor re-arrangement, we obtain
\begin{align*}
\sum_{k=1}^{n}kb_kz^{k-1}+\sum_{k=n+1}^{\infty}kb_kz^{k-1} 
& =\alpha\left(-2(1-\beta)+\sum_{k=1}^{n-1}(k-2(1-\beta))b_kz^k\right)w(z) \\
& + \alpha\left(\sum_{k=n}^{\infty}(k-2(1-\beta))b_kz^k\right)w(z). 
\end{align*}
By Clunie's method \cite{Clu59} (see also \cite{CK60,Rob70,Rog43}), 
we obtain
$$\sum_{k=1}^{n}k^2|b_k|^2{\rho}^{2k-2}\leq {\alpha}^2\left(4(1-\beta)^2
+\sum_{k=1}^{n-1}\left(k-2(1-\beta)\right)^2|b_k|^2{\rho}^{2k}\right),
$$
equivalently,
\begin{equation}\label{lem-eq1}
\sum_{k=1}^{n} k^2 |b_k|^2{\rho}^{2k-2}
 -\alpha^2 \sum_{k=1}^{n-1} (k-2(1-\beta))^2  |b_k|^2{\rho}^{2k}
 \leq 4(1-\beta)^2 \alpha^2.
\end{equation}
If we take $\rho=1$ and allow $n\rightarrow \infty$, then we obtain the desired inequality
$$\sum_{k=1}^{\infty} \left(k^2-(k-2(1-\beta))^2 \alpha^2 \right)|b_k|^2
\leq 4(1-\beta)^2 \alpha^2.
$$
The proof of our lemma is now complete.
\end{proof}

We remark that the special choices $\alpha=1$ and $\beta=0$ turned Lemma~\ref{lem1} into the well-known
Area Theorem for $f\in \es$ (see for instance \cite[Theorem~11, p. 193 of Vol-2]{Good83}).

%%%%%%%%%%%%%%%%%%%%%%%%%%%%%%%%%%%%%%%%%%%%%%%%%%%%%%%%

We now prepare a lemma using a new technique introduced in \cite{OPW13} and this lemma plays an 
important role to prove our main theorem in this paper.

\begin{lemma}\label{lem2}
Let $0<\alpha\leq 1,\, 0\leq \beta<1$, and $f(z)\in \es (\alpha,\beta)$. For $|z|<\rho\le 1$ suppose that
$$ \frac{z}{f(z)}= 1+ \sum_{k=1}^{\infty}b_kz^k ~\mbox{and}~\ (1+\alpha z)^{2-2\beta}
=1+ \sum_{k=1}^{\infty}{(-1)}^kc_kz^k. 
$$
Then the relation
\begin{equation}\label{eq5}
\sum_{k=1}^N k|b_k|^2 {\rho}^{2k} \leq \sum_{k=1}^{N}k|c_k|^2 {\rho}^{2k}
\end{equation}
is valid for all $N\in \mathbb{N}$.
\end{lemma}
\begin{proof}
We divide our proof into three steps.

{\bf Step-I: Clunie's method.}

Rewrite (\ref{lem-eq1}) in the following form:
$$
\sum_{k=1}^{n-1} \left(k^2-(k-2(1-\beta))^2 \alpha^2{\rho}^2 \right)|b_k|^2{\rho}^{2k-2}+n^2|b_n|^2{\rho}^{2n-2}
\leq 4(1-\beta)^2 \alpha^2.
$$
Multiply by $\rho^2$ on both sides we obtain 
\begin{equation}\label{eq6}
\sum_{k=1}^{n-1} \left(k^2-(k-2(1-\beta))^2 \alpha^2{\rho}^2 \right)|b_k|^2{\rho}^{2k}+n^2|b_n|^2{\rho}^{2n}
\leq 4(1-\beta)^2\alpha^2 \rho^2.
\end{equation}
The function $b(z)=(1+\alpha z)^{2-2\beta}$ clearly shows that the equality in $(\ref{eq6})$
attains with $b_k=(-1)^kc_k.$

{\bf Step-II: Cramer's Rule.}

We consider the inequalities (\ref{eq6}) for $n = 1,\ldots,N$, and multiply the $n$-th coefficient
by a factor $\lambda_{n,N}$ for each $n$. These factors are chosen in such a way that the addition of the left sides
of the modified inequalities is equivalent to the left side of (\ref{eq5}). The calculation of the
factors $\lambda_{n,N}$ leads to the following system of linear equations:
\begin{equation}\label{eq7}
k=k^2\lambda_{k,N}+\sum_{n=k+1}^{N} {\lambda}_{n,N} \left(k^2-(k-2(1-\beta))^2 \alpha^2{\rho}^2 \right),
\, k = 1, \cdots , N.
\end{equation}
Since the matrix of the system (\ref{eq7}) is an upper triangular matrix with positive integers
as diagonal elements, the solution of the system is uniquely determined.
Cramer's rule allows us to write the solution of the system (\ref{eq7}) in the form
$$\lambda_{n,N}=\frac{((n-1)\ !)^2}{(N\ ! )^2} 
\mbox{Det$~A_{n,N}$},
$$
where $A_{n,N}$ is the $(N-n+1)\times(N-n+1)$ matrix constructed as follows:
$$A_{n,N}=\left[\begin{array}{cccc}
n & n^2-(n-2(1-\beta))^2\alpha^2 \rho^2 & \cdots & n^2-(n-2(1-\beta))^2\alpha^2 \rho^2\\
n+1 & (n+1)^2 & \cdots & (n+1)^2 -(n+1-2(1-\beta))^2\alpha^2 \rho^2\\
\vdots & \vdots & \vdots & \vdots \\
N & 0 & \cdots & N^2\\
\end{array} \right]
$$
If we replace $\alpha^2 \rho^2$ by $r^2$ and $2(1-\beta)$ by $\gamma$ in (\ref{eq7}), 
then the equation is equivalent to \cite[(8)]{OPW13}. The rest of the proof now similarly follows as explained 
in \cite{OPW13}. 
\end{proof}
%%%%%%%%%%%%%%%%%%%%%%%%%%%%%%%%%%%%%%%%%%%%%%%%%%%%%%%%

We now establish a preliminary result concerning necessary and sufficient conditions for a function 
to be in $\es (\alpha,\beta).$
\begin{lemma}\label{lem3}
Let $0<\alpha\leq 1 $ and $0\leq \beta <1$. Then $f(z)\in \es (\alpha,\beta)$ if and only if $ F$
 defined by $F(z)=z\left({f(z)}/{z}\right)^{\frac{1}{1-\beta}} \in \es (\alpha), \,z\in \D.$
\end{lemma}
\begin{proof}
Let $F(z)=z\left({f(z)}/{z}\right)^{\frac{1}{1-\beta}}.$
Taking logarithm derivative on both sides and simplify, we get
$$
z\frac{F'(z)}{F(z)}=1+\frac{1}{1-\beta} \left(z\frac{f'(z)}{f(z)}-1\right).
$$
By componendo and dividendo rule, we have 
\begin{equation}\label{eq4}
 \left |\left(z\frac{F'(z)}{F(z)}-1\right){\Big /} \left(z\frac{F'(z)}{F(z)}+1\right) \right|
 =\left |\left(z\frac{f	'(z)}{f(z)}-1\right){\Big /} \left(z\frac{f'(z)}{f(z)}+1-2\beta\right) \right|.
\end{equation}
By Definition~\ref{def2} we get $f(z)\in \es (\alpha,\beta)$ if and only if 
$F(z)\in \es (\alpha)$.
\end{proof}
%%%%%%%%%%%%%%%%%%%%%%%%%%%%%%%%%%%%%%%%%%%%%%%%%%%%%%%%%%%%%%%%%%%%%
%%%%%%%%%%%%%%%%%%%%%%% Section 3 %%%%%%%%%%%%%%%%%%%%%%%%%%%%%%%%%%

\section{\bf Proof of the Main Theorem}  

As an initial observation, from (\ref{eq2}) we see that
$$\frac{z}{k_{\alpha,\beta}(z)} =(1-\alpha z)^{2(1-\beta)}
= 1+\sum_{n=1}^\infty c_nz^n
$$
where $\displaystyle c_n= \frac{(\zeta)_n}{(1)_n} {\alpha}^n$ and $\zeta =-2(1-\beta).$

Hence, we apply the area formula for the function $\displaystyle z/k_{\alpha,\beta}(z)$ and obtain
\begin{align*}
\pi^{-1}\Delta\left(\rho,\frac{z}{k_{\alpha,\beta}}\right) & =
\sum_{n=1}^{\infty}n|c_n|^2 {\rho}^{2n}, \quad |z|<\rho \\
 & = \sum_{n=1}^{\infty}n \frac{(\zeta)_n(\zeta)_n}{(1)_n(1)_n}{\alpha}^{2n}{\rho}^{2n} \\
& = {\zeta}^2{\alpha}^2{\rho}^2\sum_{n=0}^{\infty} 
\frac{(\zeta+1 )_n(\zeta+1)_n}{(2)_n(1)_n}{\alpha}^{2n}{\rho}^{2n} \\ 
& =  4{\alpha}^2(1-\beta)^2{\rho}^2 {}_2F_1(2\beta-1,2\beta-1;2;{\alpha}^2{\rho}^2)\\
 & = {\pi}^{-1}A_{\alpha,\beta}(\rho).
\end{align*}
At this point let us write $A_{\alpha,\beta}(\rho)$, $0<\rho\leq 1$, in the following form:
$$A_{\alpha,\beta}(\rho)=4\pi {\alpha}^{2}(1-\beta)^2 {\rho}^{2}
\sum_{n=0}^{\infty}\frac{(2\beta-1)_{n}^2}{(1)_{n}(2)_n}{\alpha}^{2n}{\rho}^{2n}.
$$
Because the series on the right hand side has positive coefficients, $A_{\alpha,\beta}(\rho)$ is a non-decreasing and convex
function of the real variable $\rho.$ Thus, $A_{\alpha,\beta}(\rho)\leq A_{\alpha,\beta}(1)$, i.e.
$$A_{\alpha,\beta}(\rho)\leq 4\pi {\alpha}^{2}(1-\beta)^2
\sum_{n=0}^{\infty}\frac{(2\beta-1)_{n}^2}{(1)_{n}(2)_n}{\alpha}^{2n}.
$$
In order to look for the Dirichlet-finite function $g_{\alpha,\beta}(z)=z/k_{\alpha,\beta}(z)$,
we now comprise Table~1 for the values of $A_{\alpha,\beta}(1)$ for various choices of $\alpha$ and $\beta$, 
and add pictures of the images of the unit disk under the extremal functions
$g_{\alpha,\beta}(z)$ for the corresponding 
values of $\alpha$ and $\beta$ (see Figures~2 to 3).

\begin{center}
\begin{tabular}{|c|c|c|c|}
\hline
\textbf{Values of $\beta$} & \textbf{Values of $\alpha$} & \textbf{ Approximate Values of $A_{\alpha,\beta}(1)$} \\
\hline
 2/3 & 1/4& 0.0875754 \\ \cline{2-3}
 & 2/3 & 0.638452\\ \cline{2-3}
 & 5/6 & 1.01889\\
\hline
 4/5 &1/4 & 0.0317791 \\ \cline{2-3}
  & 2/3 & 0.245872 \\ \cline{2-3}
  & 5/6 & 0.415385 \\ 
\hline
\end{tabular}
\medskip
\hspace{7.5cm} Table~1
\end{center}

\begin{figure}
\begin{minipage}[b]{0.45\textwidth}
\includegraphics[width=6.5cm]{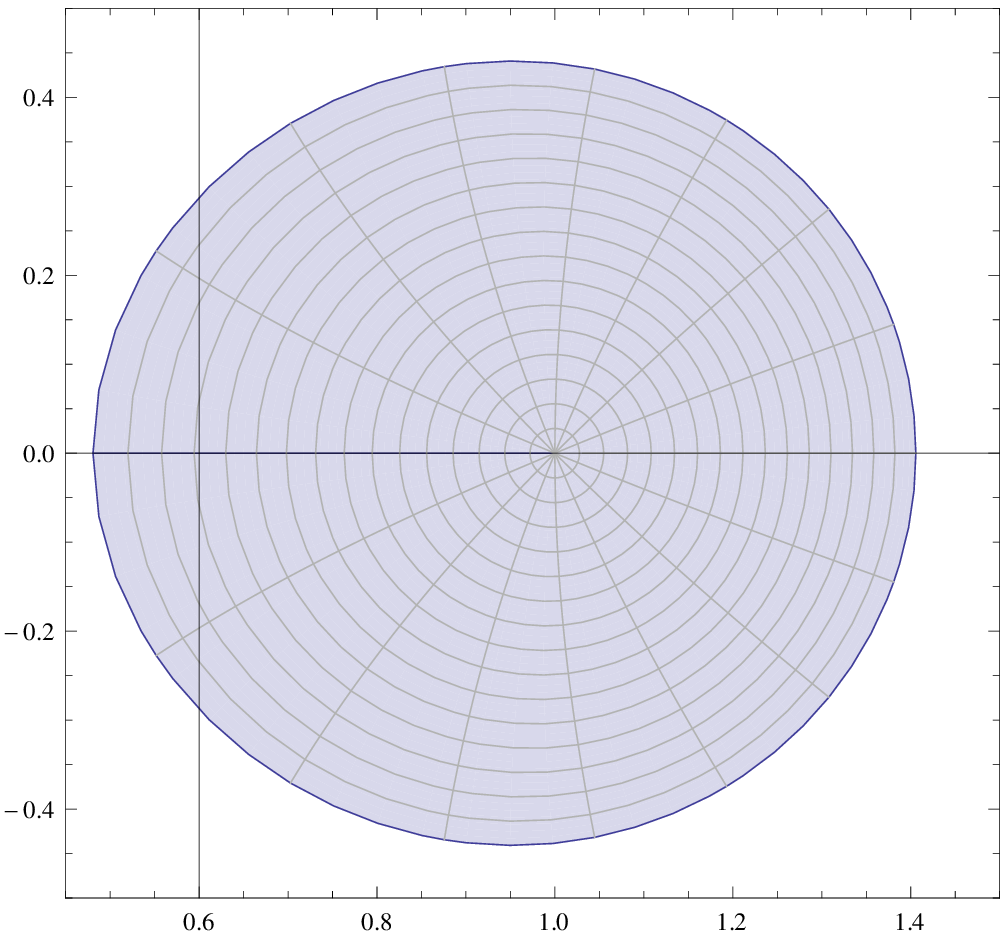}
\hspace*{1cm}The image domain $g_{2/3,2/3}(\D)$
\end{minipage}
\begin{minipage}[b]{0.45\textwidth}
\includegraphics[width=6.45cm]{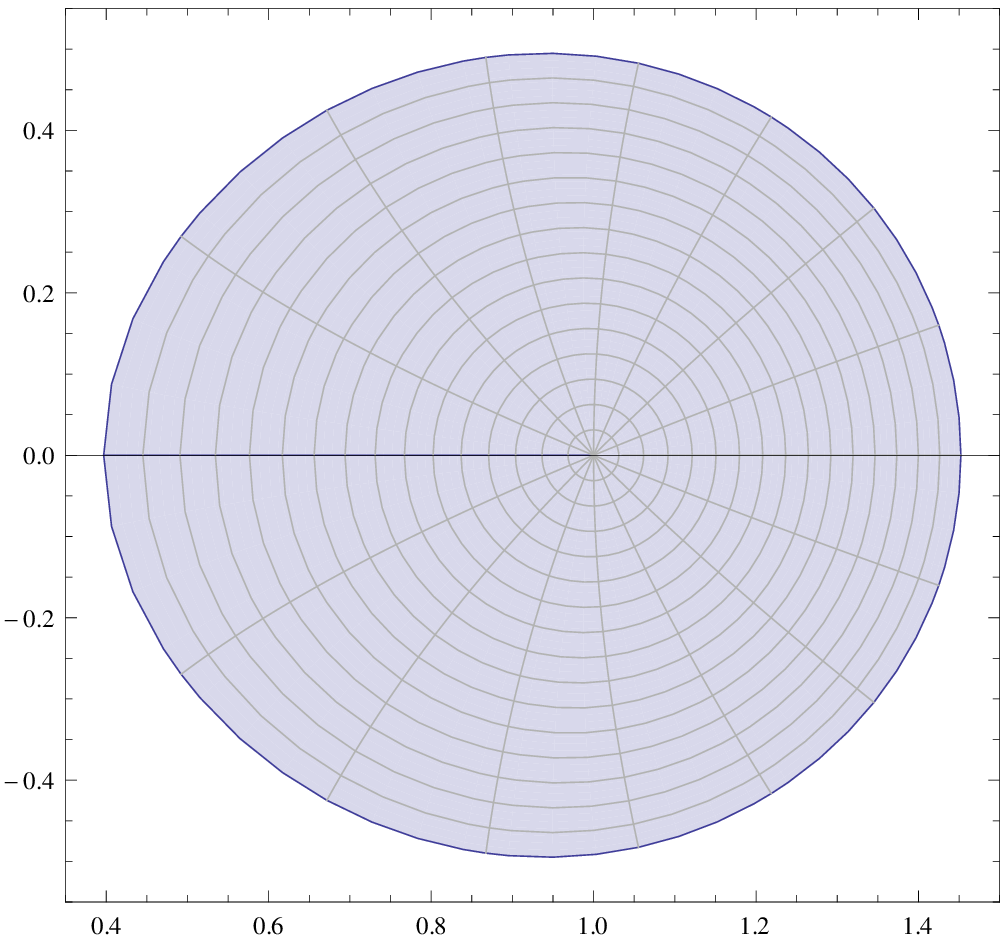}
\hspace*{1cm} The image domain $g_{5/6,2/3}(\D _{0.9})$
\end{minipage}
\caption{Images of the disks $\D$ and $\D_{0.9}$ under $g_{2/3,2/3}$ and $g_{5/6,2/3}$.}
\end{figure}

\begin{figure}
\begin{minipage}[b]{0.45\textwidth}
\includegraphics[width=6.5cm]{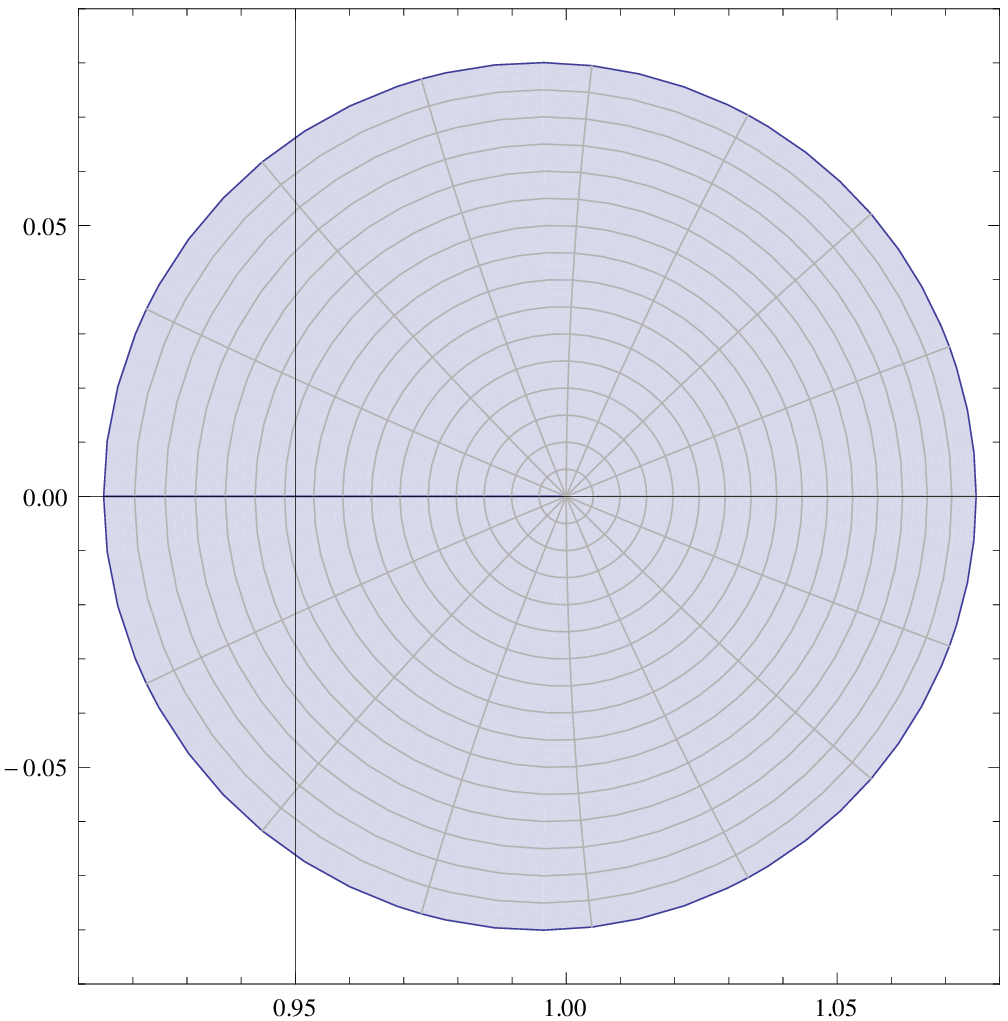}
\hspace*{1cm}The image domain $g_{1/4,4/5}(\D _{0.8})$
\end{minipage}
\begin{minipage}[b]{0.45\textwidth}
\includegraphics[width=7.2cm]{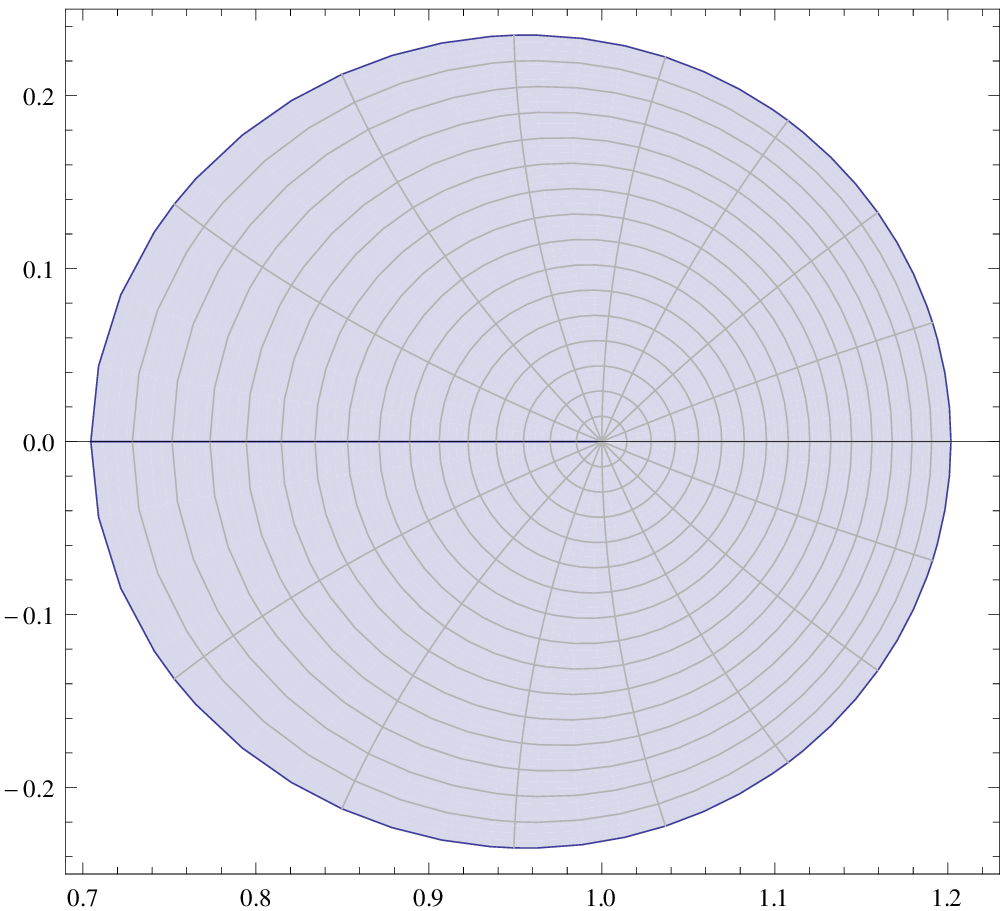}
\hspace*{1cm} The image domain $g_{5/6,4/5}(\D _{0.7})$
\end{minipage}
\caption{Images of the disks $\D_{0.8}$ and $\D_{0.7}$ under $g_{1/4,4/5}$ and $g_{5/6,4/5}$.}
\end{figure}

It is now time for us to prove the main theorem.
\begin{proof}
For $z\in \D,$ we know by Lemma \ref{lem3} that
$$f(z)\in \es (\alpha,\beta) \Longleftrightarrow 
F(z)=z\left( \frac{f(z)}{z}\right)^{\frac{1}{1-\beta}} \in \es (\alpha).
$$
Further $F(z) \in \es (\alpha)$ gives    
$$
\frac{z}{F(z)}=\left(\frac{z}{f(z)}\right)^{\frac{1}{1-\beta}} \prec (1-\alpha z)^2
=\left(\frac{z}{k_{\alpha,\beta}(z)}\right)^{\frac{1}{1-\beta}},
$$
i.e.
$$
\frac{z}{f(z)} \prec (1-\alpha z)^{2(1-\beta)}= \frac{z}{k_{\alpha,\beta}(z)}.
$$
If
$$\frac{z}{f(z)}= 1+ \sum_{n=1}^{\infty}b_nz^n ~\mbox{and}~\ \frac{z}{k(z)}
=1+ \sum_{n=1}^{\infty}c_nz^n,\quad |z|<\rho,
$$
then the extension of Rogosinski's result observed by Goluzin \cite[Theorem 6.3, p.193]{Dur83} yields
$$\sum_{n=1}^{\infty}n|b_n|^2 {\rho}^{2n} \leq \sum_{n=1}^{\infty}n|c_n|^2 {\rho}^{2n}.
$$
 
That is
$$\Delta \left(\rho,\frac{z}{f}\right) \leq \Delta \left(\rho,\frac{z}{k_{\alpha,\beta}}\right)
= \pi {\zeta}^2{\alpha}^2{\rho}^2 {}_2F_1(\zeta +1,\zeta+1;2;{\alpha}^2{\rho}^2),\quad \zeta+1=2\beta-1
$$
whenever the sequence $\{n{\rho}^{2n}\}$ is decreasing function of $\rho,\, 0<\rho\leq \frac{1}{\sqrt{2}}$.
Thus, the theorem is obviously true for $0<\rho\leq \frac{1}{\sqrt{2}}$. On other hand, in order to present 
a proof to include the case $ \rho>\frac{1}{\sqrt{2}}$, it suffices to prove
$$\sum_{n=1}^N n|b_n|^2 {\rho}^{2n} \leq \sum_{n=1}^{N}n|c_n|^2 {\rho}^{2n},\,
N\in \mathbb{N},\, \rho\in (0,1).
$$
This follows from  Lemma~\ref{lem2} and hence the proof of Theorem~\ref{thm1} is complete.
\end{proof}
%%%%%%%%%%%%%%%%%%%%%%%%%%%%%%%%%%%%%%%%%%%%%%%%%%%%%%%%%%%%

If we choose ${\beta=0}$ in Theorem~\ref{thm1}, then we get the following Yamashita conjecture problem solved for functions
in the Padmanabhan class $\es(\alpha)$: 
\begin{theorem}\label{thm2}
Let for $0<\alpha \leq 1,\, f\in \es (\alpha)$ and $z/f(z)$ be a non-vanishing analytic function in $\D$.
Then  we have 
$$\displaystyle \max_{f\in \es (\alpha)}\Delta \left(\rho,\frac{z}{f}\right)=2\pi {\alpha}^2\rho^2(2+{\alpha}^2\rho^2)
     =: A_{\alpha}(\rho)	
$$
for all $\rho,\, 0<\rho\leq 1$. The maximum is attained only by the rotation of the function $k_{\alpha}(z)$
defined by $(\ref{eq3}).$
\end{theorem}
We now collect the values of $A_{\alpha}(1)$ in the form of a table (see Table~2) for several values of $\alpha$ 
and affix geometrical pictures of the images of the unit disk under the extremal functions
$g_{\alpha}(z)=z/k_{\alpha}(z)=(1-\alpha z)^{2}$ for the corresponding $\alpha$ values 
(see Figure~4).

\begin{figure}
\begin{minipage}[b]{0.45\textwidth}
\includegraphics[width=6.5cm]{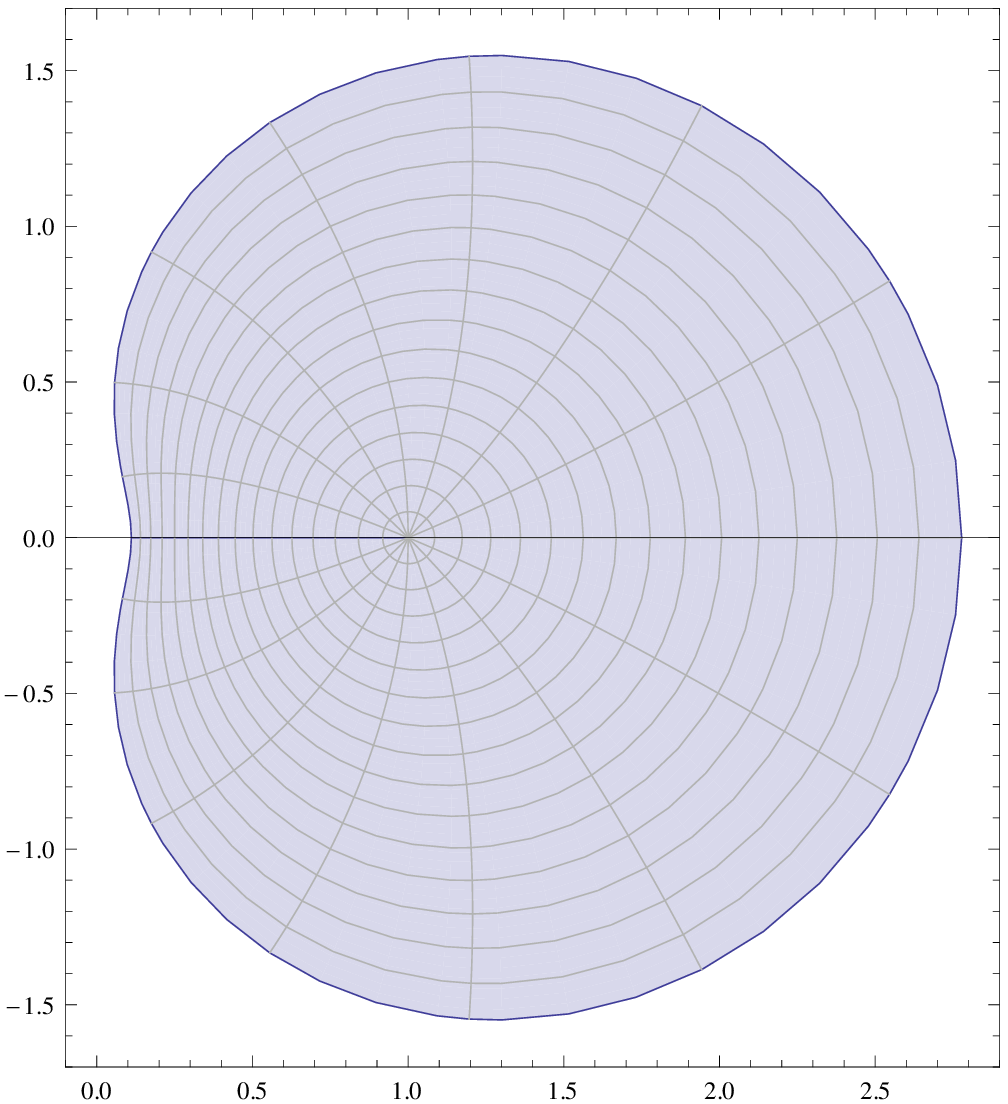}
\hspace*{1cm}The image domain $g_{2/3}(\D)$
\end{minipage}
\begin{minipage}[b]{0.45\textwidth}
\includegraphics[width=6.7cm]{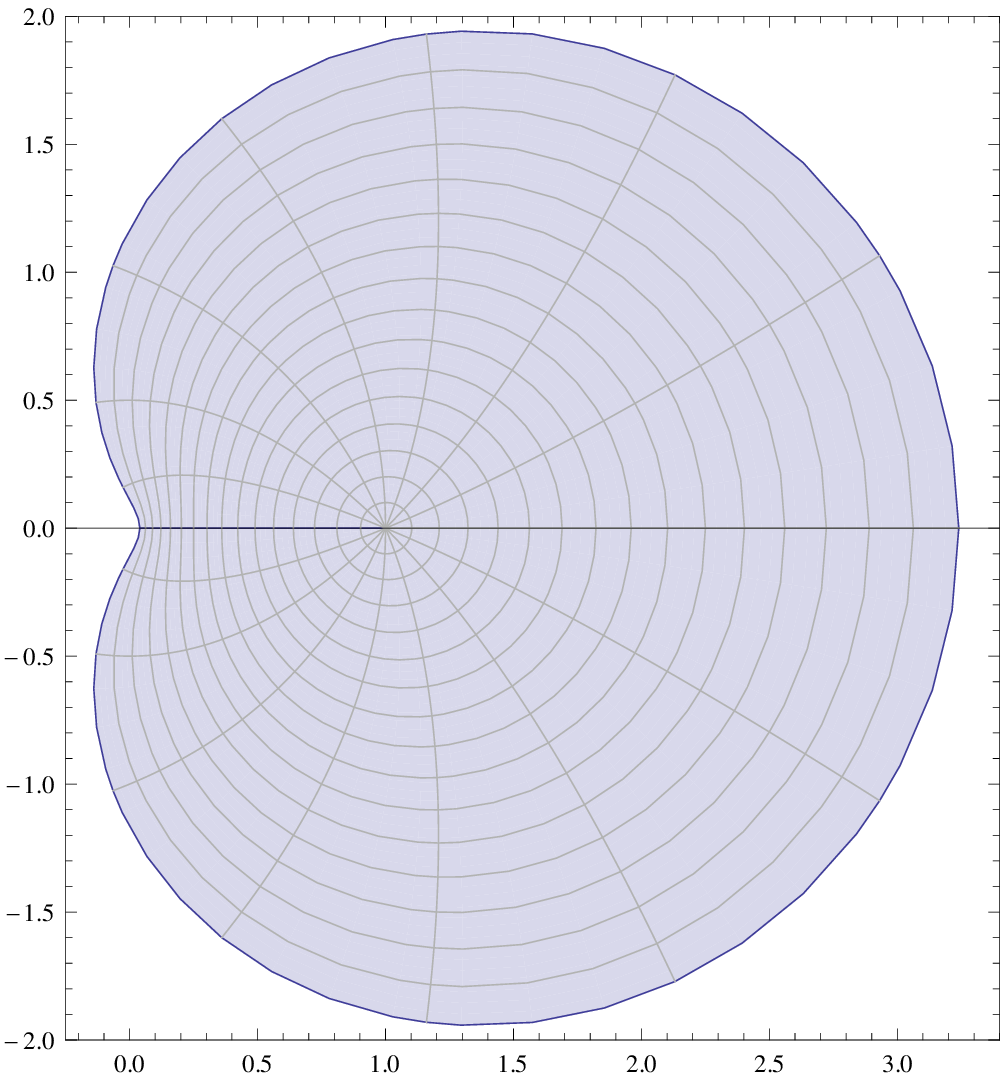}
\hspace*{1cm} The image domain $g_{8/9}(\D _{0.9})$
\end{minipage}
\caption{Images of the disks $\D$ and $\D_{0.9}$ under $g_{2/3}$ and $g_{8/9}$.}
\end{figure}

%%%%%%%%%%%%%%%%%%%%%%%%%%%%%%%%%%%%%%%%%%%%%%%%%%%%%%%%%%%%%%%%%%
\vspace*{0.5cm}
\begin{center}
\begin{tabular}{|c|c|c|c|}
\hline
\textbf{Values of $\alpha$} & \textbf{ Approximate Values of $A_{\alpha}(1)$} \\
\hline
  1/4 & 0.809942   \\ \cline{1-2}
  2/3 & 6.82616\\ \cline{1-2}
  5/6 & 11.7567\\ \cline{1-2}
  8/9 & 13.8515 \\
\hline
\end{tabular}
\medskip
\hspace{7.5cm} Table~2
\end{center}

If ${\alpha=1}$ and $\beta=1/2$, then as a consequence of Theorem~\ref{thm1} we get 
\begin{corollary}\label{cor2}\cite[Theorem 2]{OPW13}
We have 
 $$\displaystyle \max_{f\in \es t(1/2)}~\Delta \left(\rho,\frac{z}{f}\right)
    =\pi \rho^2 
 \quad \mbox{for $0<\rho\leq 1$},
 $$
 where the maximum is attained only by the rotation of the Koebe function $k(z)$
%  =z/(1-z)^{2(1-\beta)}. $ 
defined by $(\ref{eq3}).$
\end{corollary}
%%%%%%%%%%%%%%%%%%%%%%%%%%%%%%%%%%%%%%%%%%%%%%%%%%%%%%%%%%%%%%%%%%

Moreover, if we choose ${\alpha=1}$ in Theorem~\ref{thm1}, we get
\begin{corollary}\label{cor3}\cite[Theorem 3]{OPW13}
Let $f\in \es t(\beta)$ for some $0\leq \beta <1$. 
Then we have 
 $$\displaystyle \max_{f\in \es t(\beta)}\Delta \left(\rho,\frac{z}{f}\right)
    =4\pi (1-\beta)^2\rho^2 {}_2F_1(2\beta-1,2\beta-1;2;\rho^2) 
 \quad \mbox{for $0<\rho\leq 1$},
 $$
 where the maximum is attained only by the rotation of the function $k_{\beta}(z)$
defined by $(\ref{eq3}).$
\end{corollary}
%%%%%%%%%%%%%%%%%%%%%%%%%%%%%%%%%%%%%%%%%%%%%%%%%%%%%%%%%
%%%%%%%%%%%%%%%%%%%%%%%%% Section 4 %%%%%%%%%%%%%%%%%%%%%%%
\section{\bf Concluding Remark}
For $-1\leq B<A \le 1,$ 
the Janowski class $\es ^*(A,B)$ is defined by the subordination relation
$$\es ^*(A,B):=\left \{f\in \A: \frac{zf'(z)}{f(z)}\prec \frac{1+Az}{1+Bz},\quad z\in \D  \right\}.
$$
The class $\es ^*(A,B)$ is introduced in \cite{Jan73} and studied by number of researchers in this field. 
It is evident that $\es ^*(A,B)\subset \es t$.
In \cite{PW13}, it has been reported that Yamashita's conjecture is an open problem to prove for convex functions 
of order $\beta$ and more generally, for functions in the class $\es ^*(A,B)$ and also for the class of functions $f$ for which
$zf'(z)\in \es ^*(A,B)$. In particular, the choices $A=(1-2\beta)\alpha$ and $B=-\alpha$ turn the class
$\es ^*(A,B)$ into the class $\es(\alpha,\beta)$. Therefore, a partial solution to the above open problem
has been solved in this paper.

%%%%%%%%%%%%%%%%%%%%%%%%%%%%%%%%%%%%%%%%%%%%%%%%%%%%%%%%%%%%%%%%%%
\bigskip
\noindent
{\bf Acknowledgements.} This research was started in the mid 2013 when the authors were visited ISI Chennai Centre.
We would like to thank Professors R. Pavaratham and S. Ponnusamy for useful discussion in this topic.
The second author also acknowledges the support of the National Board for Higher Mathematics, Department of Atomic Energy,
India (grant no. 2/39(20)/2010-R$\&$D-II). The authors thank one of the referees for useful comments on the paper.

\end{document}